\RequirePackage{fix-cm}
\documentclass[smallextended,referee,envcountsect]{svjour3}       
\smartqed  
\usepackage{graphicx}
\usepackage{amsmath, ntheorem}
\usepackage{amssymb}
\usepackage{stackrel}
\usepackage{marvosym, enumerate}
\usepackage{amstext, amscd, latexsym}
\usepackage{amsfonts, ragged2e}
\usepackage{mathrsfs, pgfplots, enumerate, setspace}
\usepackage{subfigure}
\usepackage{cases}

\smartqed
\usepackage{cite}
\usepackage{amstext, graphicx, amscd, latexsym, amssymb}
\usepackage{amsfonts, ragged2e}
\usepackage{pgfplots, setspace}
\usepackage{geometry}
\usepackage{latexsym,bm}
\usepackage{float}
 \usepackage{epstopdf,caption}

\newtheorem{theo}{Theorem}
\newtheorem{defi}{Definition}
\newtheorem{lem}{Lemma}
\newtheorem{rem}{Remark}

\newtheorem{assumption}{Assumption}
\newtheorem{prop}{Proposition}

\usepackage[figuresright]{rotating}
\usepackage[sort&compress,numbers]{natbib}
\usepackage{algorithm} 
\usepackage{algorithmic}  
\usepackage[linesnumbered,algo2e]{algorithm2e} 
\usepackage[colorlinks,linkcolor=blue,anchorcolor=blue,citecolor=blue]{hyperref}

\begin{document}

\title{Convergence rates analysis of Interior Bregman Gradient Method for Vector Optimization Problems}

\titlerunning{Convergence rates analysis of Interior Bregman Gradient Method for Vector Optimization Problems}        

\author{Jian Chen \and Liping Tang \and  Xinmin Yang  }

\institute{J. Chen \at Department of Mathematics, Shanghai University,
	Shanghai 200444, China\\
                    chenjian\_math@163.com\\
                   L.P. Tang \at National Center for Applied Mathematics of Chongqing, Chongqing 401331, China\\
                   tanglipings@163.com\\
        \Letter X.M. Yang \at National Center for Applied Mathematics of Chongqing, and School of Mathematical Sciences,  Chongqing Normal University, Chongqing 401331, China\\
        xmyang@cqnu.edu.cn  \\}

\date{Received: date / Accepted: date}

\maketitle

\begin{abstract}
In recent years, by using Bregman distance, the Lipschitz gradient continuity and strong convexity were lifted and replaced by relative smoothness and relative strong convexity. Under the mild assumptions, it was proved that gradient methods with Bregman regularity converge linearly for single-objective optimization problems (SOPs). In this paper, we extend the relative smoothness and relative strong convexity to vector-valued functions and analyze the convergence of an interior Bregman gradient method for vector optimization problems (VOPs). Specifically, the global convergence rates are $\mathcal{O}(\frac{1}{k})$ and $\mathcal{O}(r^{k})(0<r<1)$ for convex and relative strongly convex VOPs, respectively. Moreover, the proposed method converges linearly for VOPs that satisfy a vector Bregman-PL inequality.

\keywords{Vector optimization \and Bregman distance \and Relative smoothness \and Relative strong convexity \and Linear convergence}
\subclass{65K05 \and 90C26 \and 90C29 \and 90C30}
\end{abstract}

\section{Introduction}
Let $C\subset\mathbb{R}^{m}$ be a closed, convex, and pointed cone with a non-empty interior. The partial order induced by $C$:
$$y\preceq_{C}(\mathrm{resp.} \prec_{C})y^{\prime}\ \Leftrightarrow\ y^{\prime}-y\in C (\mathrm{resp.}\ \mathrm{int}(C)).$$
In this paper, we consider the the following vector optimization problem
\begin{align*}
	\min\limits_{x\in \Omega} F(x), \tag{VOP}\label{VOP}
\end{align*}
where every component $F_{i}:\mathbb{R}^{n}\rightarrow\mathbb{R}^{m}$ is differentiable, and $\Omega\subset\mathbb{R}^{n}$ is closed and convex with a non-empty interior.
For the optimal solutions of (\ref{VOP}), none of the objectives can be improved without sacrificing the others. The concept of optimality is thus defined as \textit{efficiency}.
It is worth noting that (\ref{VOP}) corresponds to a multiobjective optimization problem 
when $C=\mathbb{R}^{m}_{+}$, where $\mathbb{R}^{m}_{+}$ is the non-negative orthant of $\mathbb{R}^{m}$. Some applications of multiobjective optimization problems can be found in engineering \cite{1a}, economics \cite{1b,eco}, and machine learning \cite{mac1,m2,mac2}, etc. 
\par Scalarization approaches \cite{mnp,luc,joh} are widely explored for VOPs, which convert a VOP into a single-objective optimization problem so that classical numerical optimization methods can be applied to obtain the solutions to the original problem. Unfortunately, the approaches burden the decision-maker with some parameters which are unknown in advance. To overcome the drawback, Fliege and Svaiter \cite{6} invented the parameter-free method, called the steepest descent method for multiobjective optimization. The corresponding method for VOPs is described in \cite{vsd}. Motivated by the work of Fliege and Svaiter, some standard first-order method methods are extended to solve MOPs, and VOPs \cite{8,9a,9b,9c,9d,chen}. The main idea of these methods is to compute a descent direction; then, a line search technique is performed along the direction to ensure sufficient decrease for all objective functions. By the strategy, these descent methods produce a sequence of points that enjoy global convergence. Very recently, it was proved that the convergence rates of the multiobjective gradient descent method \cite{com} and proximal gradient method \cite{pcom} are the same as the ones for SOPs. 
\par Similar to the first-order method for SOPs, a widely used assumption for MOPs is that the objective function gradients have to be globally Lipschitz continuous, which guarantees a sufficient descent in each iteration. By using Bregman distance, the restrictive assumption was replaced by the Lipschitz-like condition \cite{lc}, or relative smoothness \cite{rs} for SOPs. On the other hand, the relative strong convexity was proposed to replace strong convexity in \cite{rs}. Under the mild assumptions, the linear convergence results were established for gradient \cite{rs} and proximal gradient method \cite{as} with Bregman regularization. Moreover, the linear convergence results of the gradient method were obtained with Bregman error bound condition \cite{nlns}. Auslender and Teboulle \cite{ip} presented another advantage of the Bregman method: the produced sequence can be restricted in the interior of the feasible set by choosing a suitable Bregman distance and thus eliminate the constraints. 
\par The Bregman gradient methods were also explored in VOPs. Villacorta and Oliveira \cite{mip} developed the interior Bregman gradient method with a modified convergence sensing condition for convex VOPs with a non-polyhedral set. In \cite{cz}. Chen et al. introduced the vector-valued Bregman function for MOPs, and the corresponding vector-valued Bregman distance replaced the Euclidean distance in the proximal point algorithm (PPA). Very recently, the PPA with vector-valued Bregman distance was applied to solve multiobjective DC programming in \cite{2022}. Moreover, the proximal gradient method with Bregman functions for MOPs has been investigated by Ansary and Dutta \cite{2022a}.
\par It should be noted that convergence analysis with the appealing potential of Bregman distance was not discussed in the methods mentioned above. Naturally, a question arises can we obtain linear convergence results of the Bregman gradient method for VOPs without Lipschitz gradient continuity and strong convexity. The paper's main contribution is to answer the question positively. More specifically:
\begin{itemize}
	\item[$\bullet$] We extend the relative smoothness and relative strong convexity to vector-valued functions in the context of VOPs;
	\item[$\bullet$] We present two merit functions and analyze their properties, such as optimality, continuity, and interrelation. We also proposed a generalized vector Bregman-PL inequality, which will be used to prove the linear convergence of the proposed method.
	\item[$\bullet$] We propose an interior Bregman gradient method for VOPs that restricts the produced sequence inside the feasible set. It is proved that every accumulation point of the produced sequence is $C$-stationary point for VOP, and the whole sequence converges to a weakly efficient set with $C$-convex objective functions.
	\item[$\bullet$] With relative smooth and strongly convex objective functions, we prove the produced sequence converges linearly to a weakly efficient point in the sense of Bregman distance. We also prove that a merit function converges linearly with the vector Bregman-PL inequality.
\end{itemize}
\par The organization of the paper is as follows. Some notations and definitions are given in Sect. 2 for our later use. Sect. 3 discusses the merit functions for VOPs. Sect. 4 is devoted to introducing the interior Bregman gradient method and proving the convergence results for the proposed method. At the end of the paper, some conclusions are drawn.
\section{Preliminaries}
\subsection{Notations}
\begin{itemize}
	\item[$\bullet$] $[m]=\{1,2,...,m\}$.\\
	
	\item[$\bullet$]$\Delta^{+}_{m}=\left\{\lambda:\sum\limits_{i\in[m]}\lambda_{i}=1,\lambda_{i}>0,\ i\in[m]\right\}$ the relative interior of $m$-dimensional unit simplex.\\
	 
	\item[$\bullet$]  $\|\cdot\|$ the Euclidean norm, $\langle\cdot,\cdot\rangle$ the inner product.\\
	\item[$\bullet$] $\mathrm{int}(\cdot)$ and $\mathrm{cl}(\cdot)$ the interior and the closure of a set, respectively.\\
	\item[$\bullet$] $\mathrm{conv}(\cdot)$ the convex hull of a set.\\
	\item[$\bullet$] $B[x,r]$ the closed ball centred at $x$ with radius $r$.\\
	\item[$\bullet$] $JF(x)\in\mathbb{R}^{m\times n}$ and $\nabla F_{i}(x)\in\mathbb{R}^{n}$ the Jacobian matrix and the gradient of $F_{i}$ at $x$, respectively.\\
	\item[$\bullet$] $C^{*}=\{c^{*}\in\mathbb{R}^{m}:\langle c^{*},c\rangle\geq0, \forall c\in C\}$ the positive polar cone of $C$.\\
	\item[$\bullet$] $G=\mathrm{conv}(\{c^{*}\in\mathbb{R}^{m}:\|c^{*}\|=1, \forall c^{*}\in C^{*}\})$.\\
	\item[$\bullet$] $\omega^{*}(x^{*})=\sup\limits_{x\in\mathbb{R}^{n}}\{\langle x^{*},x\rangle-\omega(x)\}$ the conjugate function of $\omega$.\\
\end{itemize}
\subsection{Vector optimization}
Recall some definitions of solutions to (\ref{VOP}) as follows.

\begin{defi}{\rm\cite{joh}}
	A vector $x^{*}\in\Omega$ is called efficient solution to (\ref{VOP}), if there exists no $x\in\Omega$ such that $F(x)\preceq_{C} F(x^{\ast})$ and $F(x)\neq F(x^{\ast})$.
\end{defi}

\begin{defi}{\rm\cite{joh}}\label{we}
	A vector $x^{*}\in\Omega$ is called weakly efficient solution to (\ref{VOP}), if there exists no $x\in\Omega$ such that $F(x)\prec_{C} F(x^{\ast})$.
\end{defi}

\begin{defi}{\rm\cite{vsd}}\label{st}
	A vector $x^{\ast}\in\Omega$ is called $C$-stationary point to (\ref{VOP}), if
	$$JF(x^{*})(\Omega-x)\cap(-\mathrm{int}(C))=\emptyset.$$
\end{defi}

For a non-stationary point $x$, there exists a descent direction that is defined as:
\begin{defi}{\rm\cite{vsd}}\label{cd}
	A vector $d\in\mathbb{R}^{n}$ is called $C$-descent direction for $F$ at $x$, if
	$$JF(x)d\in-\mathrm{int}(C).$$
\end{defi}
\begin{rem}\label{rdd}
	If $x\in\Omega$ is a non-stationary point, then there exists a vector $y\in\Omega$ such that $JF(x)(y-x)\in-\mathrm{int}(C)$. Note that $\mathrm{int}(C)$ is an open set, the relation is also valid for some $y\in\mathrm{int}(\Omega)$. It follows from Definition \ref{cd} that $y-x$ is a $C$-descent direction. 
\end{rem}

\begin{defi}{\rm\cite{joh}}
	A function $F(\cdot)$ is called $C$-convex on $\Omega$, if
	$$F(\lambda x+(1-\lambda)y)\preceq_{C}\lambda F(x)+(1-\lambda)F(y),\ \forall x,y\in\Omega,\ \lambda\in[0,1].$$
\end{defi}

Since $F(\cdot)$ is differentiable, $C$-convexity of $F(\cdot)$ on $\Omega$ is equivalent to
$$JF(x)(y-x)\preceq_{C}F(y)-F(x),\ \forall x,y\in\Omega.$$
By using the positive polar cone $C^{*}$, an equivalent characterization of $C$-convexity of $F(\cdot)$ on $\Omega$ is presented as follows.
\begin{lem}{\rm{\cite{luc}}}
	A function $F(\cdot)$ is $C$-convex on $\Omega$ if and only if $\langle c^{*},F(\cdot)\rangle$ is convex on $\Omega$ for all $c^{*}\in C^{*}$.
\end{lem}

\begin{rem}\label{rem1}
	Note that $\frac{c^{*}}{\|c^{*}\|}\in G$ for all $c^{*}\in C^{*}\setminus\{0\}$, then $F(\cdot)$ is $C$-convex on $\Omega$ if and only if $\langle c^{*},F(\cdot)\rangle$ is convex on $\Omega$ for all $c^{*}\in G$.
\end{rem}

In $C$-convex case, we derive the equivalence between $C$-stationary point and weakly efficient solution. 
\begin{lem}{\rm\cite{joh}}
	Assume that the objective function $F(\cdot)$ is $C$-convex on $\Omega$. Then $x^{*}\in\Omega$ is a $C$-stationary point of (\ref{VOP}) if and only if $x^{*}$ is an efficient solution of (\ref{VOP}).
\end{lem}

\subsection{Relative smoothness and relative strong convexity}
\begin{defi}[Legendre function]{\rm\cite{roc}}
	Let $\omega:X\rightarrow \left(-\infty,+\infty \right]$ be a proper closed convex function, where $X\subset\mathbb{R}^{n}$. Then 
	\begin{itemize}
		\item[$\mathrm{(i)}$] $\omega(\cdot)$ is essentially smooth if $\omega(\cdot)$ is differentiable on $\mathrm{int (dom} \omega)$ and $\lim\limits_{k\rightarrow\infty}\|\nabla\omega(x^{k})\|\rightarrow\infty$ for all $\{x^{k}\}\subset\mathrm{int (dom} \omega)$ converges to some boundary point of $\mathrm{dom} \omega$;
		\item[$\mathrm{(ii)}$] $\omega(\cdot)$ is a Legendre function if $\omega(\cdot)$ is essentially smooth and strictly convex on $\mathrm{int (dom} \omega)$.
	\end{itemize}
\end{defi}
\begin{lem}\cite{roc}\label{lroc}
For a Legendre function $\omega(\cdot)$, we have the following useful properties.
\begin{itemize}
	\item[$\mathrm{(i)}$] $\omega(\cdot)$ is Legendre if and only if its conjugate $\omega^{*}$ is Legendre.
	\item[$\mathrm{(ii)}$] $\omega^{*}(\nabla\omega(x))=\langle x,\nabla\omega(x)\rangle-\omega(x)$.
	\item[$\mathrm{(iii)}$] $(\nabla\omega(\cdot))^{-1}=\nabla\omega^{*}(\cdot)$.  
	\item[$\mathrm{(iv)}$] $\mathrm{dom}\partial\omega=\mathrm{int(dom}\omega)$ with $\partial\omega(x)=\{\nabla\omega(x)\},\ \ \forall x\in\mathrm{int(dom}\omega)$.
\end{itemize}
\end{lem}
Using the Legendre function $\omega(\cdot)$. the Bregman distance $D_{\omega}(\cdot,\cdot)$ w.r.t. $\omega(\cdot)$ is defined as
$$D_{\omega}(x,y)=\omega(x)-\omega(y)-\langle \nabla\omega(y),x-y\rangle,\ \forall(x,y)\in\mathrm{dom}(\omega)\times\mathrm{int(dom}\omega).$$
Since $\omega(\cdot)$ is strictly convex on $\mathrm{int (dom} \omega)$, we have $D_{\omega}(x,y)\geq0$ for all $x,y\in\mathrm{int (dom} \omega)$, and the equality holds if and only if $x=y$. In what follows, we recall a basic property for $D_{\omega}(\cdot,\cdot)$ that will be used in the sequel:
\begin{equation}\label{ed}
	D_{\omega}(x,y)=D_{\omega^{*}}(\nabla\omega(y),\nabla\omega(x)),\ \forall x,y\in\mathrm{int (dom} \omega).
\end{equation}

In general, $D_{\omega}(\cdot,\cdot)$ is not symmetric. The measure of symmetry for $D_{\omega}(\cdot,\cdot)$ was introduced in \cite{lc}. The symmetric coefficient is given by
$$\alpha(\omega):=\inf\left\{\frac{D_{\omega}(x,y)}{D_{\omega}(y,x)}:x,y\in\mathrm{int (dom} \omega),\ x\neq y\right\}\in[0,1].$$
\begin{rem}
	As noted in \cite{lc}, we call $D_{\omega}(\cdot,\cdot)$ symmetric if $\alpha(\omega)=1$. If $\alpha(\omega)=0$, then symmetry is absent for $D_{\omega}(\cdot,\cdot)$. The former happens when $\omega$ is a strictly convex quadratic function, which was deduced from \cite[Lemma 3.16]{sym}. As shown in \cite[Proposition 2]{lc}, if $\mathrm{dom}\omega$ is not open, then $\alpha(\omega)=0$.
\end{rem}

In the context of vector optimization, we propose the relative smoothness of $F(\cdot)$ relative to Legendre function $\omega(\cdot)$.
\begin{defi} For a vector $e\in\mathrm{int}C$, 
	we call $F(\cdot)$ is $(L,C,e)$-smooth relative to $\omega(\cdot)$ on $\Omega\subseteq\mathrm{dom}(\omega)$ if for all $x,y\in\mathrm{int}\Omega$, there exists $L>0$ such that
	$$F(y)\preceq_{C}F(x)+JF(x)(y-x)+LD_{\omega}(y,x)e.$$
\end{defi}

The following proposition presents some properties for relative smoothness.
\begin{prop}\label{rs}
	Assume $F(\cdot)$ is $(L,C,e)$-smooth relative to $\omega(\cdot)$ on $\Omega\subseteq\mathrm{dom}(\omega)$, then
	\begin{itemize}
		\item[$\mathrm{(i)}$] it is equivalent to $L\omega(\cdot)e-F(\cdot)$ is $C$-convex on $\Omega$;
		\item[$\mathrm{(ii)}$] for any $c\in\mathrm{int}C$, $F(\cdot)$ is $(\gamma L,C,c)$-smooth relative to $\omega(\cdot)$ on $\Omega$, where $\gamma=\max\limits_{c^{*}\in G}\left\{\frac{\langle c^{*},e\rangle}{\langle c^{*},c\rangle}\right\}$.
	\end{itemize}
\end{prop}
\begin{proof}
(i) From Remark \ref{rem1}, $C$-convexity of $L\omega(\cdot)e-F(\cdot)$ on $\Omega$ is equivalent to $\langle c^{*}, L\omega(\cdot)e-F(\cdot)\rangle$ is convex on $\Omega$ for all $c^{*}\in G$. Then, we have
$$\langle c^{*}, LD_{\omega}(y,x)e-F(y)+F(x)+JF(x)(y-x)\rangle\geq0,\ \forall x,y\in\mathrm{int}(\Omega),\ c^{*}\in G.$$
This is equivalent to $F(y)\preceq_{C}F(x)+JF(x)(y-x)+LD_{\omega}(y,x)e$, assertion (i) is proved.
\par (ii) From assertion (i), we obtain $\langle c^{*}, L\omega(\cdot)e-F(\cdot)\rangle$ is convex on $\Omega$ for all $c^{*}\in G$. Since $\omega(\cdot)$ is convex, for any $c\in\mathrm{int}C$, it follows that $\langle c^{*}, \bar{L}\omega(\cdot)c-F(\cdot)\rangle$ is convex on $\Omega$ for all $c^{*}\in G$ when $\bar{L}\langle c^{*}, c\rangle\geq L\langle c^{*}, e\rangle$ holds for all $c^{*}\in G$. The inequality holds by setting $\bar{L}=\max\limits_{c^{*}\in G}\left\{\frac{\langle c^{*},e\rangle}{\langle c^{*},c\rangle}\right\}L$, and $\max\limits_{c^{*}\in G}\left\{\frac{\langle c^{*},e\rangle}{\langle c^{*},c\rangle}\right\}$ is well-defined due to $c\in\mathrm{int}C$ and the compactness of $G$.$\ \ \square$
\end{proof}

We also propose the relative strong convex of $F(\cdot)$ relative to Legendre function $\omega(\cdot)$.
\begin{defi}
	For a vector $e\in\mathrm{int}C$, 
	we call $F(\cdot)$ is $(\mu,C,e)$-strongly convex relative to $\omega(\cdot)$ on $\Omega\subseteq\mathrm{dom}(\omega)$ if for all $x,y\in\mathrm{int}\Omega$, there exists $\mu\geq0$ such that
	$$F(x)+JF(x)(y-x)+\mu D_{\omega}(y,x)e\preceq_{C}F(y).$$
\end{defi}

Similarly, we present some properties for relative strong convexity as follows. 
\begin{prop}
	Assume $F(\cdot)$ is $(\mu,C,e)$-strongly convex relative to $\omega(\cdot)$ on $\Omega\subseteq\mathrm{dom}(\omega)$, then
	\begin{itemize}
		\item[$\mathrm{(i)}$] it is equivalent to $F(\cdot)-\mu\omega(\cdot)e$ is $C$-convex on $\Omega$;
		\item[$\mathrm{(ii)}$] for any $c\in\mathrm{int}C$, $F(\cdot)$ is $(\rho\mu,C,c)$-strongly convex relative to $\omega(\cdot)$ on $\Omega$, where $\rho=\min\limits_{c^{*}\in G}\left\{\frac{\langle c^{*},e\rangle}{\langle c^{*},c\rangle}\right\}$.
	\end{itemize}
\end{prop}
\begin{proof}
	The proof is similar to the arguments used in Proposition \ref{rs}, we omit it here.$\ \ \square$
	\end{proof}

\begin{rem}
	When $\omega(\cdot)=\frac{1}{2}\|\cdot\|^{2}$, the $(L,C,e)$-smoothness reduces to condition (A) in \cite{cw}. If $C=\mathbb{R}_{+}$, the $(L,C,e)$-smoothness and $(\mu,C,e)$-strong convexity correspond to relative $L$-smoothness and $\mu$-strong convexity in \cite{rs}, respectively.
\end{rem}

From the above two propositions, the constants of relative smoothness and relative strong convexity depend on $\omega(\cdot)$, $F(\cdot)$ and $e$. Recall that the quotient $\frac{L}{\mu}$ plays a key role in the geometric convergence of first-order methods in the presence of strong convexity. If $\frac{L}{\mu}$ is large, the function descreases slowly when first-order methods are applied. We call the function is ill-conditioned in the situation. To alleviate this dilemma, for the general $\omega(\cdot)$ and $F(\cdot)$, the vector $e$ is selected as:
\begin{equation}\label{e}
	e:=\mathop{\arg\max}\limits_{c\in B[0,r]}\min\limits_{c^{*}\in G}\langle c^{*},c\rangle.
\end{equation}
	When $C=\mathbb{R}_{+}^{m}$, the vector $e$ defined as (\ref{e}) corresponds to $(\frac{r}{\sqrt{m}},...,\frac{r}{\sqrt{m}})$. For simplicity, we select a suitable $r>0$ such that 
	\begin{equation}\label{max}
		\max\limits_{c^{*}\in G}\langle c^{*},e\rangle=1.
	\end{equation}
On the other hand, since $C^{*}\subset\mathbb{R}^{m}$ is a closed, convex, and pointed cone, we have $0\notin G$. Then we defined
	\begin{equation}\label{min}
		\delta:=\min\limits_{c^{*}\in G}\langle c^{*},e\rangle>0.
	\end{equation}
\section{ Merit Functions for VOPs}
A function is called merit function for (\ref{VOP}) if
it returns $0$ at the solutions of (\ref{VOP}) and strictly positive values otherwise. In this section, we first present the following two merit functions. 
\begin{equation}\label{u}
	u_{0}(x):=\sup\limits_{y\in\Omega}\min\limits_{c^{*}\in G}\langle c^{*},F(x)-F(y)\rangle,
\end{equation}

\begin{equation}\label{v}
	v_{\ell}(x):=\sup\limits_{y\in\Omega}\min\limits_{c^{*}\in G}\{\langle c^{*},JF(x)(x-y)\rangle-\ell D_{\omega}(y,x)\},
\end{equation}
where $\ell>0$ is a constant and $\mathrm{cl}(\mathrm{dom}\omega)=\Omega$. 
\subsection{Properties of merit functions}
\par We can show that $u_{0}(\cdot)$ and $v_{\ell}(\cdot)$ are merit functions in the sense of weak efficiency and stationarity, respectively.
\begin{theo}\label{t1}
	Let $u_{0}(\cdot)$ and $v_{\ell}(\cdot)$ be defined as (\ref{u}) and (\ref{v}), respectively. Then
	\begin{itemize}
		\item[$\mathrm{(i)}$] the point $x\in\Omega$ is a weakly efficient solution of (\ref{VOP}) if and only if $u_{0}(x)=0$;
		\item[$\mathrm{(ii)}$] the point $x\in\mathrm{int}(\Omega)$ is $C$-stationary point of (\ref{VOP}) if and only if $v_{\ell}(x)=0$;
	\end{itemize}
\end{theo}
\begin{proof}
	\par (i) Since  $x$ is a weakly efficient solution of (\ref{VOP}), we have
	$$F(x)-F(y)\notin \mathrm{int}(C),\ \forall y\in\Omega.$$
	Then $$\min\limits_{c^{*}\in G}\langle c^{*},F(x)-F(y)\rangle\leq0,\ \forall y\in\Omega,$$
	which is equivalent to $$\sup\limits_{y\in\Omega}\min\limits_{c^{*}\in G}\langle c^{*},F(x)-F(y)\rangle\leq0.$$
	On the other hand, the definition of $u_{0}(\cdot)$ implies that
	$$u_{0}(x)\geq\min\limits_{c^{*}\in G}\{\langle c^{*},F(x)-F(x)\rangle\}=0,$$
	which is equivalent to $u_{0}(x)=0$. 
	\par (ii) Since  $x$ is $C$-stationary point of (\ref{VOP}), we have
	$$\sup\limits_{y\in\Omega}\min\limits_{c^{*}\in G}\langle c^{*},JF(x)(x-y)\rangle\leq0,$$
	which, along the fact that $D_{\omega}(y,x)\geq0$, yields
	$$	v_{\ell}(x)=\sup\limits_{y\in\Omega}\min\limits_{c^{*}\in G}\{\langle c^{*},JF(x)(x-y)\rangle-\ell D_{\omega}(y,x)\}\leq0.$$
	On the other hand, the definition of $v_{\ell}$ implies that
	$$v_{\ell}(x)\geq\min\limits_{c^{*}\in G}\{\langle c^{*},JF(x)(x-x)\rangle-\ell D_{\omega}(x,x)\}=0.$$
	Hence, $v_{\ell}(x)=0$. 
\par Conversely, assume that $x$ is not $C$-stationary when $v_{\ell}(x)=0$. Then, there exists $y_{0}\in\Omega$ such that 
	$$\min\limits_{c^{*}\in G}\langle c^{*},JF(x)(x-y_{0})\rangle>0.$$
	For all $t\in[0,1]$, $x+t(y_{0}-x)\in\Omega$. Hence,
	\begin{align*}
		v_{\ell}(x)&=\sup\limits_{y\in\Omega}\min\limits_{c^{*}\in G}\{\langle c^{*},JF(x)(x-y)\rangle-\ell D_{\omega}(y,x)\}\\
		&\geq\min\limits_{c^{*}\in G}\{\langle c^{*},tJF(x)(x-y_{0})\rangle-\ell D_{\omega}(x+t(y_{0}-x),x)\}.
	\end{align*}
On the other hand, the differentiability of $\omega(\cdot)$ implies $\lim\limits_{t\rightarrow0}\frac{D_{\omega}(x+t(y_{0}-x),x)}{t}=0$. This together with the above inequality and $\min\limits_{c^{*}\in G}\langle c^{*},JF(x)(x-y_{0})\rangle>0$ yield $v_{\ell}(x)>0$. This contradicts $v_{\ell}(x)=0$. The proof is completed.$\ \ \square$
\end{proof}

We denote by $V_{\ell}(\cdot)$ the optimal solution of $v_{\ell}(\cdot)$. Hence,
\begin{equation}\label{V}
	V_{\ell}(x):=\mathop{\arg\sup}\limits_{y\in\Omega}\min\limits_{c^{*}\in G}\{\langle c^{*},JF(x)(x-y)\rangle-\ell D_{\omega}(y,x)\}.
\end{equation}
In the sequel, we systematically assume that $V_{\ell}(\cdot)$ is well-defined on $\mathrm{int(dom}\omega)$, i.e., $V_{\ell}(\cdot)$ is a single valued mapping from $\mathrm{int(dom}\omega)$ to $\mathrm{int(dom}\omega)$. We give a sufficient condition for the assumption.
 \begin{lem}\label{net}
 	If $\omega(\cdot)$ is supercoercive, i.e., $$\lim\limits_{\|x\|\rightarrow\infty}\frac{\omega(x)}{\|x\|}=\infty,$$
 	then $V_{\ell}(\cdot)$ is a single valued mapping from $\mathrm{int(dom}\omega)$ to $\mathrm{int(dom}\omega)$.
 \end{lem}
\begin{proof}
For any fixed point $x\in\mathrm{int(dom}\omega)$, and $\ell>0$. We define the function $P_{\ell}(\cdot)$ for $u$ as follows
\begin{equation}\label{p}
	P_{\ell}(u)=\min\limits_{c^{*}\in G}\{\langle c^{*},JF(x)(x-u)\rangle-\ell D_{\omega}(u,x)\}.
\end{equation}
Then $V_{\ell}(x)=\mathop{\arg\sup}\limits_{u\in\Omega}P_{\ell}(u)$. Substituting the equality of Bregman distance into equality (\ref{p}), we have
\begin{align*}
	P_{\ell}(u)&=\min\limits_{c^{*}\in G}\{\langle c^{*},JF(x)(x-u)\rangle-\ell(\omega(u)-\omega(x)-\nabla\omega(x)(u-x))\}\\
	&=\|u\|\left(\frac{\min\limits_{c^{*}\in G}\{\langle c^{*},JF(x)(x-u)\rangle\}}{\|u\|}-\frac{\ell\omega(u)}{\|u\|}+\frac{\ell(\omega(x)+\nabla\omega(x)(u-x))}{\|u\|}\right).
\end{align*}	
Taking $\|u\|\rightarrow\infty$, then $\frac{\min\limits_{c^{*}\in G}\{\langle c^{*},JF(x)(x-u)\rangle\}}{\|u\|}$ is bounded due to the compactness of $G$. This combines with the boundness of $\frac{\ell(\omega(x)+\nabla\omega(x)(u-x))}{\|u\|}$ and supercoercivity of $\omega(\cdot)$ implies $\lim\limits_{\|u\|\rightarrow\infty}P_{\ell}(u)=-\infty$. Hence, $P_{\ell}(\cdot)$ is level bounded on $\mathbb{R}^{n}$. It follows from continuity of $P_{\ell}(\cdot)$ and Weierstrass's theorem \cite[Theorem 1.9]{rocv} that $V_{\ell}(x)$ is nonempty. The uniqueness of $V_{\ell}(x)$ is given by the strict convexity of $\omega(\cdot)$. From the optimality condition of $V_{\ell}(x)$, it follows that $\partial\omega(V_{\ell}(x))$ is nonempty. Then Lemma \ref{lroc}(iv) implies that $V_{\ell}(x)\in\mathrm{int(dom}\omega)$.$\ \ \square$
\end{proof}
\begin{rem}
	Recall that $\mathrm{cl}(\mathrm{dom}\omega)=\Omega$, we have $\mathrm{int(dom}\omega)=\mathrm{int}(\Omega)$. From Lemma \ref{net}, if $\omega(\cdot)$ is supercoercive, then $V_{\ell}(\cdot)$ is a single valued mapping from $\mathrm{int}(\Omega)$ to $\mathrm{int}(\Omega)$. 
\end{rem}

\begin{prop}\label{p1}
	For all $\ell>0$, let $v_{\ell}(\cdot)$ and $V_{\ell}(\cdot)$ be defined as (\ref{v}) and (\ref{V}), respectively. Then
	\begin{itemize}
		\item[$\mathrm{(i)}$] there exists $c_{\ell}(x)\in G$, such that 
		\begin{equation}\label{e3}
			V_{\ell}(x)=\mathop{\arg\sup}\limits_{y\in\Omega}\{\langle c_{\ell}(x),JF(x)(x-y)\rangle-\ell D_{\omega}(y,x)\},	
		\end{equation}
	    and
	    \begin{equation}\label{ecl}
	    	c_{\ell}(x)\in\mathop{\arg\min}\limits_{c^{*}\in G}\langle c^{*},JF(x)(x-V_{\ell}(x))\rangle;
	    \end{equation}
		\item[$\mathrm{(ii)}$] $V_{\ell}(x)=\nabla\omega^{*}(\nabla\omega(x)-\frac{1}{\ell}JF(x)^{T}c_{\ell}(x)),\ \forall x\in\mathrm{int}(\Omega);$
		\item[$\mathrm{(iii)}$] $\langle c_{\ell}(x),JF(x)(x-V_{\ell}(x))\rangle=\ell(D_{\omega}(x,V_{\ell}(x))+D_{\omega}(V_{\ell}(x),x)),\ \forall x\in\mathrm{int}(\Omega);$
		\item[$\mathrm{(iv)}$] $v_{\ell}(x)=\ell D_{\omega}(x,V_{\ell}(x)),\ \forall x\in\mathrm{int}(\Omega)$.
	\end{itemize} 
\end{prop}
\begin{proof}
	(i) Note that $\Omega$ is convex, $G$ is compact and convex, and $\langle c^{*},JF(x)(x-y)\rangle-\ell D_{\omega}(y,x)$ is convex for $c^{*}$ and concave for $y$. Therefore, it follows by Sion's minimax theorem \cite{32} that there exists $c_{\ell}(x)\in G$ such that
	\begin{align*}
		&\ \sup\limits_{y\in\Omega}\min\limits_{c^{*}\in G}\{\langle c^{*},JF(x)(x-y)\rangle-\ell D_{\omega}(y,x)\}\\
		=&\ \min\limits_{c^{*}\in G}\sup\limits_{y\in\Omega}\{\langle c^{*},JF(x)(x-y)\rangle-\ell D_{\omega}(y,x)\}\\
		=&\ \langle c_{\ell}(x),JF(x)(x-V_{\ell}(x))\rangle-\ell D_{\omega}(V_{\ell}(x),x),
	\end{align*}
and
	$$V_{\ell}(x)\in\mathop{\arg\sup}\limits_{y\in\Omega}\{\langle c_{\ell}(x),JF(x)(x-y)\rangle-\ell D_{\omega}(y,x)\},$$
and
	$$c_{\ell}(x)\in\mathop{\arg\min}\limits_{c^{*}\in G}\langle c^{*},JF(x)(x-V_{\ell}(x))\rangle.$$
Hence, we obtain (\ref{ecl}), and (\ref{e3}) follows by the strict convexity of $\omega(\cdot)$.	
	\par (ii) From the optimality condition for (\ref{e3}) and the fact that $V_{\ell}(x)\in\mathrm{int}(\Omega)$, we have
	\begin{equation}\label{e4}
	-JF(x)^{T}c_{\ell}(x)-\ell(\nabla\omega(V_{\ell}(x))-\nabla\omega(x))=0.	
	\end{equation}
Then, assertion (i) follows from the fact that $(\nabla\omega(\cdot))^{-1}=\nabla\omega^{*}(\cdot)$. 
\par (iii) From (\ref{e4}), we have
\begin{align*}
	\langle c_{\ell}(x),JF(x)(x-V_{\ell}(x))\rangle&=\ell\langle\nabla\omega(V_{\ell}(x))-\nabla\omega(x),V_{\ell}(x)-x\rangle\\
	&=\ell(D_{\omega}(V_{\ell}(x),x)+D_{\omega}(x,V_{\ell}(x))).
\end{align*}
Hence, we obtain the desired result.
\par (iv) Using (\ref{e3}) and the fact that $V_{\ell}(x)$ is the unique solution of $v_{\ell}(x)$, we obtain
	\begin{align*}
		v_{\ell}(x)&=\langle c_{\ell}(x),JF(x)(x-V_{\ell}(x))\rangle-\ell D_{\omega}(V_{\ell}(x),x)\\
		&=\ell(D_{\omega}(V_{\ell}(x),x)+D_{\omega}(x,V_{\ell}(x)))-\ell D_{\omega}(V_{\ell}(x),x)\\
		&=\ell D_{\omega}(x,V_{\ell}(x)),
	\end{align*}
where the second equality is given by assertion (ii). The proof is completed.$\ \ \square$
\end{proof}

We can also show the continuity of $v_{\ell}(\cdot)$ and $V_{\ell}(\cdot)$.
\begin{prop}
	For all $\ell>0$, the $v_{\ell}(\cdot)$ and $V_{\ell}(\cdot)$ are continuous on $\mathrm{int}(\Omega)$.
\end{prop}
\begin{proof}
	From Proposition \ref{p1}(iv), it is sufficient to prove the continuity of $V_{\ell}(\cdot)$.
	For all $x\in \mathrm{int}(\Omega)$, the optimality condition (\ref{e4}) gives
	$$\ell\langle\nabla\omega(x) - \nabla\omega(V_{\ell}(x)), z-V_{\ell}(x) \rangle=\langle JF(x)^{T}c_{\ell}(x),z-V_{\ell}(x)\rangle,\ \forall z\in \Omega.$$
	For $\bar{x}\in\mathrm{int}(\Omega)$ and a sequence $\{x_{k}\}\subset\mathrm{int}(\Omega)$ satisfying $x_{k}\rightarrow\bar{x}$. 
	Substituting $(x,z)=(\bar{x},V_{\ell}(x_{k}))$ and $(x,z)=(x_{k},V_{\ell}(\bar{x}))$ into the above equality, respectively, and sum them up, we have
	\begin{align*}
		&\ell\langle\nabla\omega(x_{k})-\nabla\omega(\bar{x})+\nabla\omega(V_{\ell}(\bar{x}))-\nabla\omega(V_{\ell}(x_{k})),V_{\ell}(\bar{x})-V_{\ell}(x_{k}) \rangle\\
		=&\langle JF(x_{k})^{T}c_{\ell}(x_{k})-JF(\bar{x})^{T}c_{\ell}(\bar{x}),V_{\ell}(\bar{x})-V_{\ell}(x_{k})\rangle\\
		=&\langle JF(x_{k})^{T}c_{\ell}(x_{k}),x_{k}-V_{\ell}(x_{k})\rangle +	\langle JF(\bar{x})^{T}c_{\ell}(\bar{x}),\bar{x}-V_{\ell}(\bar{x})\rangle\\
		&+	\langle JF(x_{k})^{T}c_{\ell}(x_{k}),V_{\ell}(\bar{x})-x_{k}\rangle+\langle JF(\bar{x})^{T}c_{\ell}(\bar{x}),V_{\ell}(x_{k})-\bar{x}\rangle\\
		\leq&\langle JF(x_{k})^{T}c_{\ell}(\bar{x}),x_{k}-V_{\ell}(x_{k})\rangle +	\langle JF(\bar{x})^{T}c_{\ell}(x_{k}),\bar{x}-V_{\ell}(\bar{x})\rangle\\
		&+	\langle JF(x_{k})^{T}c_{\ell}(x_{k}),V_{\ell}(\bar{x})-x_{k}\rangle+\langle JF(\bar{x})^{T}c_{\ell}(\bar{x}),V_{\ell}(x_{k})-\bar{x}\rangle\\
		=&\langle(JF(x_{k})-JF(\bar{x}))^{T}c_{\ell}(\bar{x}),x_{k}-V_{\ell}(x_{k})\rangle +	\langle(JF(\bar{x})-JF(x_{k}))^{T}c_{\ell}(x_{k}),\bar{x}-V_{\ell}(\bar{x})\rangle\\		
		&+	\langle JF(x_{k})^{T}c_{\ell}(x_{k}),\bar{x}-x_{k}\rangle+\langle JF(\bar{x})^{T}c_{\ell}(\bar{x}),x_{k}-\bar{x}\rangle,
	\end{align*}
	where the inequality is given by (\ref{ecl}).
	For simplicity, we denote the right hand side of the above inequality as $RHS$.
	Then, we have
	\begin{align*}
		\ell\langle\nabla\omega(V_{\ell}(\bar{x}))-\nabla\omega(V_{\ell}(x_{k})),V_{\ell}(\bar{x})-V_{\ell}(x_{k}) \rangle
		\leq RHS+\langle\nabla\omega(\bar{x})-\nabla\omega(x_{k}),V_{\ell}(\bar{x})-V_{\ell}(x_{k})\rangle.
	\end{align*}
	Notice that $\omega(\cdot)$ is strictly convex on $\mathrm{int}(\Omega)$, we have
	$$\langle\nabla\omega(V_{\ell}(\bar{x}))-\nabla\omega(V_{\ell}(x_{k})),V_{\ell}(\bar{x})-V_{\ell}(x_{k}) \rangle\geq0.$$
	On the other hand, since $\nabla\omega(\cdot)$ and $JF(\cdot)$ are continuous, and $x_{k}\rightarrow \bar{x}$, we obtain $RHS$ and $\nabla\omega(\bar{x})-\nabla\omega(x_{k})$ tend to $0$. Then $\langle\nabla\omega(V_{\ell}(\bar{x}))-\nabla\omega(V_{\ell}(x_{k})),V_{\ell}(\bar{x})-V_{\ell}(x_{k}) \rangle$ tends to $0$. It follows by the strict convexity of $\omega$ that $V_{\ell}(x_{k})\rightarrow V_{\ell}(\bar{x})$. Hence, the continuity of $\omega(\cdot)$ follows from the arbitrary of $\{x_{k}\}$ and $\bar{x}\in\mathrm{int}(\Omega)$.$\ \ \square$
\end{proof}

\subsection{Vector Bregman-PL inequality}
In this subsection, a noval vector Bregman-PL inequality is derived with merit functions. 
At first, we present the connection between $D_{\omega}(x,\nabla\omega^{*}(\nabla\omega(x)-\ell_{1}JF(x)^{T}c_{\ell}(x)))$ and $D_{\omega}(x,\nabla\omega^{*}(\nabla\omega(x)-\ell_{2}JF(x)^{T}c_{\ell}(x)))$ for some $\ell_{1},\ell_{2}>0$, which will be used to establish the connection between $v_{\ell_{1}}(x)$ and $v_{\ell_{2}}(x)$.
\begin{lem}\label{leq}
	Let $\ell_{1},\ell_{2}>0$, $c^{*}\in G$ and $x\in\mathrm{int}(\Omega)$. If $\ell_{1}\geq\ell_{2}$, we have
	\begin{equation}\label{e10}
		D_{\omega}(x,\nabla\omega^{*}(\nabla\omega(x)-\ell_{1}JF(x)^{T}c^{*}))\geq D_{\omega}(x,\nabla\omega^{*}(\nabla\omega(x)-\ell_{2}JF(x)^{T}c^{*})).
	\end{equation}
\end{lem}
\begin{proof}
	The proof is similar to the arguments in the proof of \cite[Lemma 3.5]{nlns}, we omit it here.$\ \ \square$
\end{proof}

Note that it is not clear whether the inequality holds for $0<\ell_{1}<\ell_{2}$ in Lemma \ref{leq}. To overcome this difficulty we present the following assumption.
\begin{assumption}\label{a1}
	Let $\ell_{1},\ell_{2}>0$, $c^{*}\in G$ and $x\in\mathrm{int}(\Omega)$. Then there exists $\theta:\mathbb{R}_{++}\rightarrow\mathbb{R}_{++}$ such that
	\begin{equation}\label{e9}
		D_{\omega}(x,\nabla\omega^{*}(\nabla\omega(x)-\ell_{1}JF(x)^{T}c^{*}))\geq\theta({\frac{\ell_{1}}{\ell_{2}}})	D_{\omega}(x,\nabla\omega^{*}(\nabla\omega(x)-\ell_{2}JF(x)^{T}c^{*}))
	\end{equation}
\end{assumption}

The Assumption \ref{a1} seems strict in general. Indeed, we have the following sufficient condition for Assumption \ref{a1}, which is first presented in \cite{nlns}.
\begin{rem}
	\begin{itemize}
		\item[$\mathrm{(i)}$] If $\ell_{1}\geq\ell_{2}$, Assumption \ref{a1} holds with $\theta(\cdot)\equiv1$.
		\item[$\mathrm{(ii)}$] Assumption \ref{a1} holds when $\omega(\cdot)$ is $\sigma$-strongly convex and $\kappa$-smooth. In this setting, the conjugate $\omega^{*}(\cdot)$ satisfies
		$$\frac{\|u-v\|^{2}}{2\kappa}\leq D_{\omega^{*}}(u,v)\leq\frac{\|u-v\|^{2}}{2\sigma},\ \forall u,v\in\mathrm{int}(\Omega).$$
	\end{itemize} 
It follows by (\ref{ed}) and the above inequalities that
$$D_{\omega}(x,\nabla\omega^{*}(\nabla\omega(x)-\ell_{1}JF(x)^{T}c^{*}))\geq\frac{\ell_{1}^{2}}{2\kappa}\|JF(x)^{T}c^{*}\|^{2},$$
and
$$\|JF(x)^{T}c^{*}\|^{2}\geq \frac{2\sigma}{\ell_{2}^{2}}D_{\omega}(x,\nabla\omega^{*}(\nabla\omega(x)-\ell_{2}JF(x)^{T}c^{*})).$$
Then Assumption \ref{a1} holds with $\theta(t):=\frac{\sigma t^{2}}{\kappa}$.
\end{rem}

We are now in a position to present the relation between $v_{\ell_{1}}(\cdot)$ and $v_{\ell_{2}}(\cdot)$.
\begin{prop}\label{eb1}
	Suppose Assumption \ref{a1} holds. Let $\ell_{1},\ell_{2}>0$, and $x\in \mathrm{int}(\Omega)$. For $\ell_{1}\geq\ell_{2}$, we have
	\begin{equation}\label{e11}
		v_{\ell_{1}}(x)\leq v_{\ell_{2}}(x)\leq \frac{\ell_{2}}{\theta(\frac{\ell_{2}}{\ell_{1}})\ell_{1}}v_{\ell_{1}}(x).
	\end{equation}
\end{prop}
\begin{proof}
	The left hand side of inequalities (\ref{e11}) follows directly from the definition of $v_{\ell}(\cdot)$. Next, we proof the right hand side of inequalities (\ref{e11}).
	From (\ref{e3}) and the definition of $c_{\ell}(\cdot)$, we have
	$$v_{\ell_{2}}(x)\leq\sup\limits_{y\in\Omega}\{\langle c_{\ell_{1}}(x),JF(x)(x-y)\rangle-\ell_{2} D_{\omega}(y,x)\}.$$
	Using the similar argument in the proof of Proposition \ref{p1}(iv), we obtain
	\begin{align*}
		v_{\ell_{2}}(x)\leq\sup\limits_{y\in\Omega}\{\langle c_{\ell_{1}}(x),JF(x)(x-y)\rangle-\ell_{2} D_{\omega}(y,x)\}=\ell_{2}D_{\omega}(x,\nabla\omega^{*}(\nabla\omega(x)-\frac{1}{\ell_{2}}JF(x)^{T}c_{\ell_{1}}(x))).
	\end{align*}	
We further use inequality (\ref{e9}) to get
\begin{align*}
	v_{\ell_{2}}(x)&\leq\frac{\ell_{2}}{\theta(\frac{\ell_{2}}{\ell_{1}})}D_{\omega}(x,\nabla\omega^{*}(\nabla\omega(x)-\frac{1}{\ell_{1}}JF(x)^{T}c_{\ell_{1}}(x)))\\
	&=\frac{\ell_{2}}{\theta(\frac{\ell_{2}}{\ell_{1}})\ell_{1}}v_{\ell_{1}}(x),
\end{align*}
where the equality is given by Proposition \ref{p1}(iv). Therefore, we obtain the right hand side of inequalities (\ref{e11}).$\ \ \square$
\end{proof}

Moreover, we present the relations between $u_{0}(\cdot)$ and $v_{\ell}(\cdot)$.
\begin{prop}\label{eb2}
	For the merit functions $u_{0}(\cdot)$ and $v_{\ell}(\cdot)$, the following statements hold.
	\begin{itemize}
		\item[$\mathrm{(i)}$] If $F(\cdot)$ is $(L,C,e)$-smooth relative to $\omega(\cdot)$ for some $L>0$, then 
		\begin{equation}\label{evu}
			v_{L}(x)\leq u_{0}(x),\ \forall x\in\mathrm{int}(\Omega). 
		\end{equation}
		\item[$\mathrm{(ii)}$] If $F(\cdot)$ is $(\mu,C,e)$-strongly convex relative to $\omega(\cdot)$ for some $\mu\geq0$, then 
		\begin{equation}\label{e12}
			u_{0}(x)\leq v_{\delta\mu}(x),\ \forall x\in\mathrm{int}(\Omega). 
		\end{equation}
	\end{itemize} 
\end{prop}
\begin{proof}
	(i) From the $(L,C,e)$-smoothness of $F(\cdot)$, we have
	$$F(y)\preceq_{C}F(x)+JF(x)(y-x)+LD_{\omega}(y,x)e,\ \forall y\in\Omega.$$
	Hence, 
	\begin{align*}
		\sup\limits_{y\in\Omega}\min\limits_{c^{*}\in G}\{\langle c^{*},F(x)-F(y)\rangle\}&
		\geq\sup\limits_{y\in\Omega}\min\limits_{c^{*}\in G}\{\langle c^{*},JF(x)(x-y)-L D_{\omega}(y,x)e\rangle\}\\
		&\geq\sup\limits_{y\in\Omega}\{\min\limits_{c^{*}\in G}\{\langle c^{*},JF(x)(x-y)\rangle\}+
		\min\limits_{c^{*}\in G}\{\langle c^{*},-L D_{\omega}(y,x)e\rangle\}\}\\
		&\geq\sup\limits_{y\in\Omega}\min\limits_{c^{*}\in G}\{\langle c^{*},JF(x)(x-y)\rangle- LD_{\omega}(y,x)\},
	\end{align*}
where the last inequality is given by (\ref{max}). It follows that $u_{0}(x)\geq v_{L}(x)$.
\par From the $(\mu,C,e)$-strong comvexity of $F(\cdot)$, we have
$$F(x)+JF(x)(y-x)+\mu D_{\omega}(y,x)e\preceq_{C}F(y).$$
Therefore, 
\begin{align*}
	\sup\limits_{y\in\Omega}\min\limits_{c^{*}\in G}\{\langle c^{*},F(x)-F(y)\rangle\}&
	\leq\sup\limits_{y\in\Omega}\min\limits_{c^{*}\in G}\{\langle c^{*},JF(x)(x-y)-\mu D_{\omega}(y,x)e\rangle\}\\
	&\leq\sup\limits_{y\in\Omega}\{\min\limits_{c^{*}\in G}\{\langle c^{*},JF(x)(x-y)\rangle\}+
	\min\limits_{c^{*}\in G}\{\langle c^{*},-\mu D_{\omega}(y,x)e\rangle\}\}\\
	&\leq\sup\limits_{y\in\Omega}\min\limits_{c^{*}\in G}\{\langle c^{*},JF(x)(x-y)\rangle- \delta\mu D_{\omega}(y,x)\},
\end{align*}
where the last inequality is given by (\ref{min}). Hence, $u_{0}(x)\leq v_{\delta\mu}(x)$.$\ \ \square$
\end{proof}

We propose the following error bound property, which will be used in convergence rate analysis.
\begin{defi}
	We say that vector Bregman-PL inequality holds for $F(\cdot)$ on $\mathrm{int}(\Omega)$ when
	\begin{equation}\label{eb}
		u_{0}(x)\leq \tau(\ell)v_{\ell}(x),\ \forall x\in\mathrm{int}(\Omega), \ell>0,
	\end{equation}	
where $\tau:\mathbb{R}_{++}\rightarrow\mathbb{R}_{++}$.
\end{defi}
\begin{rem}
	The vector Bregman-PL inequality (\ref{eb}) corresponds to the general PL inequality when $m=1$ and $\omega(\cdot)=\frac{1}{2}\|\cdot\|^{2}$. It is also a generalization for gradient dominated inequality \cite[Definition 3.2]{nlns} in the context of vector optimization problems. Recently, a multiobjective proximal-PL inequality was established in \cite[Definition 4.1]{pcom}. The vector Bregman-PL inequality coincides with multiobjective proximal-PL inequality when $\omega(\cdot)=\frac{1}{2}\|\cdot\|^{2}$ and $g_{i}(\cdot)=\sigma_{\Omega}(\cdot)$ for all $i\in[m]$.
\end{rem}

By virtue of Propositions \ref{eb1} and \ref{eb2}, we give the sufficient conditions for vector Bregman-PL inequality.
\begin{prop}
	Suppose that Assumption \ref{a1} holds and $F(\cdot)$ is $(\mu,C,e)$-strongly convex relative to $\omega(\cdot)$ for some $\mu\geq0$. Then for all $x\in\mathrm{int}(\Omega)$ and $\ell>0$, vector Bregman-PL inequality (\ref{eb}) holds with $\tau(\ell)=\frac{\delta\mu}{\theta(\frac{\delta\mu}{\ell})\ell}$.
\end{prop}
\begin{proof}
	The result follows directly from inequalities (\ref{e11}) and (\ref{e12}).$\ \ \square$
\end{proof}

\section{ Interior Bregman Gradient Method for VOPs}
In this section, we first introduce interior Bregman gradient method for VOPs. 
\begin{algorithm} 
	\caption{{\ttfamily{interior\_Bregman\_gradient\_method\_for\_VOPs}}}\label{A1}
	\SetAlgoLined  
	\KwData{$x^{0}\in\mathrm{int}(\Omega),\ \ell>0$,\ a supercoercive Legendre function $\omega$ with $\Omega=\mathrm{cl}(\mathrm{dom}\omega)$ and $F(\cdot)$ is $(L,C,e)$-smooth related to $\omega(\cdot)$.}
	\For{$k=0,1,...$}{  $x^{k+1}=V_{\ell}(x^{k})$  \\
		\If{$x^{k+1}=x^{k}$}{ {\bf{return}} Pareto critical point $x^{k}$ }}  
\end{algorithm}

The proposed algorithm generates a sequence via the following iterates:
$$x^{k+1}=V_{\ell}(x^{k}).$$
From Lemma \ref{net} and the strict convexity of $\omega(\cdot)$, the algorithm is well-defined, and the generated sequence $\{x^{k}\}$ lies in $\mathrm{int}(\Omega)$. 
Note that $V_{\ell}(x^{k})$ is the unique solution of 
$$v_{\ell}(x^{k})=\sup\limits_{y\in\Omega}\min\limits_{c^{*}\in G}\{\langle c^{*},JF(x^{k})(x^{k}-y)\rangle-\ell D_{\omega}(y,x^{k})\},$$
then
$$V_{\ell}(x^{k})=\mathop{\arg\inf}\limits_{y\in\Omega}\max\limits_{c^{*}\in G}\{\langle c^{*},JF(x^{k})(y-x^{k})\rangle+\ell D_{\omega}(y,x^{k})\}.$$
Define mapping $\varphi:\mathbb{R}^{m}\rightarrow\mathbb{R}$ as $$\varphi(F(x)):=\max\limits_{c^{*}\in G}\{\langle c^{*},F(x)\rangle.$$
From the iterates in Algorithm \ref{A1}, we deduce the following sufficient descent property.
\begin{lem}[Sufficient descent property]\label{l1}
	Assume that $F(\cdot)$ is $(L,C,e)$-smooth relative to $\omega(\cdot)$, and $\ell>0$. Then the sequence $\{x^{k}\}$ produced by Algorithm \ref{A1} satisfies
	\begin{equation}\label{e5}
		\varphi(F(x^{k+1})-F(x^{k}))\leq-((1+\alpha(\omega))\ell-L)D_{\omega}(x^{k+1},x^{k}).
	\end{equation}
Furthermore, the sufficient descent property holds when
$$\ell>\frac{L}{1+\alpha(\omega)}.$$
\end{lem} 
\begin{proof}
From the the definition of $(L,C,e)$-smoothness, we have
\begin{align*}
	F(x^{k+1})-F(x^{k})\preceq_{C}JF(x^{k})(x^{k+1}-x^{k})+LD_{\omega}(x^{k+1},x^{k})e.
\end{align*}
Then 
\begin{align*}
	\varphi(F(x^{k+1})-F(x^{k}))&\leq \varphi(JF(x^{k})(x^{k+1}-x^{k}))+LD_{\omega}(x^{k+1},x^{k})\\
	&=-\ell D_{\omega}(x^{k+1},x^{k}) - \ell D_{\omega}(x^{k},x^{k+1})+LD_{\omega}(x^{k+1},x^{k})\\
	&\leq-((1+\alpha(\omega))\ell-L)D_{\omega}(x^{k+1},x^{k}),
\end{align*}
where the first equality follows from $x^{k+1}=V_{\ell}(x^{k})$, the definitions of $\varphi(\cdot)$, relation (\ref{ecl}) and Proposition \ref{p1}(iii), the second equality is given by the definition of $\alpha(\omega)$. This together with the inequality $\ell>\frac{L}{1+\alpha(\omega)}$ yield the sufficient descent property.$\ \ \square$
\end{proof}

We can see that Algorithm \ref{A1} terminates with an $C$-stationary point in a finite number of iterations or generates an infinite sequence. In the sequel, we will suppose that Algorithm \ref{A1} generates an infinite sequence of nonstationary points. Firstly, we present the global convergence of Algorithm \ref{A1} in nonconvex case.
\begin{theo}\label{t2}
	Assume that $\{x:F(x)\preceq_{C} F(x^{0})\}$ is bounded, $F(\cdot)$ is $(L,C,e)$-smooth relative to $\omega(\cdot)$ for some $L>0$, and $\ell>\frac{L}{1+\alpha(\omega)}$. Let $\{x^{k}\}$ be the sequence generated by Algorithm \ref{A1}. Then, we have
	\begin{itemize}
		\item[$\mathrm{(i)}$] there exists $F^{*}\preceq_{C}F(x^{k})$ for all $k$ such that $\lim\limits_{k\rightarrow\infty}F(x^{k})=F^{*}$.
		\item[$\mathrm{(ii)}$] $\sum\limits_{k=0}^{\infty}D_{\omega}(x^{k+1},x^{k})<\infty$.
		\item[$\mathrm{(iii)}$] $\{x_{k}\}$ has at least one accumulation point, and every accumulation point $x^{*}\in\mathrm{int}(\Omega)$ is a $C$-stationary point. Moreover, if $\nabla\omega(\cdot)$ is $L_{\omega}$-Lipschitz continuous on $\mathrm{int}(\Omega)$, every accumulation point $x^{*}\in\mathrm{bd}(\Omega)$ is a $C$-stationary point.
		\item[$\mathrm{(iv)}$] $\min\limits_{0\leq p\leq k-1}v_{\ell}(x^{p})\leq\frac{-\ell\varphi(F^{*}-F(x^{0}))}
		{\alpha(\omega)((1+\alpha(\omega))\ell-L)k}$
	\end{itemize}
\end{theo} 
\begin{proof}
	(i) From Lemma \ref{l1}, it follows that $\{F(x_{k})\}$ is descreasing under the partial order induced by $C$. This combined with boundness of $\{x:F(x)\preceq_{C} F(x^{0})\}$ implies $\{x_{k}\}$ is bounded and there exists $F^{*}\preceq_{C}F(x^{k})$ for all $k$ such that $\lim\limits_{k\rightarrow\infty}F(x^{k})=F^{*}$.
	\par (ii) Summing the inequality (\ref{e5}) over $k=0,1,...p$, we obtain
	\begin{align*}
	-\sum\limits_{k=0}^{p}((1+\alpha(\omega))\ell-L)D_{\omega}(x^{k+1},x^{k})&\geq\sum\limits_{k=0}^{p}	\varphi(F(x^{k+1})-F(x^{k}))\\
	&\geq\varphi(\sum\limits_{k=0}^{p}\{F(x^{k+1})-F(x^{k})\})\\
	&\geq\varphi(F^{*}-F(x^{0}))
	\end{align*}
where the second inequality follows from the subadditiveness of $\varphi(\cdot)$, and the last inequality is given by the definition of $\varphi(\cdot)$ and $F^{*}\preceq_{C}F(x^{k})$. Hence
$$\sum\limits_{k=0}^{\infty}D_{\omega}(x^{k+1},x^{k})<\infty$$
\par (iii) From Proposition \ref{p1}(iv), we have
\begin{equation}\label{e6}
	v_{\ell}(x^{k})=\ell D_{\omega}(x^{k},x^{k+1})\leq\frac{\ell}{\alpha(\omega)}D_{\omega}(x^{k+1},x^{k}),
\end{equation} 
this together with assertion (ii) imply 
$$\lim\limits_{k\rightarrow\infty}v_{\ell}(x^{k})=0.$$
On the other hand, the boundness of $\{x^{k}\}$ yields that there exists a subsequence $\{x^{k}\}_{k\in\mathcal{K}}\subset\{x^{k}\}$ tends to $x^{*}$, where $x^{*}$ is an accumulation point. Next, we prove $x^{*}$ is a $C$-stationary point. We distinguish two cases: $x^{*}\in\mathrm{int}(\Omega)$ or $x^{*}\in\mathrm{bd}(\Omega)$. If $x^{*}\in\mathrm{int}(\Omega)$, it follows from the continuity of $v_{\ell}(\cdot)$ that
$$v_{\ell}(x^{*})=\lim\limits_{k\in\mathcal{K}}v_{\ell}(x^{k})=0.$$
Therefore, $x^{*}$ is a $C$-stationary point. If  $x^{*}\in\mathrm{bd}(\Omega)$, suppose by contrary that $x^{*}$ is a non-stationary point. From Remark \ref{rdd}, there exists $y_{0}\in\mathrm{int}(\Omega)$ such that 
$$JF(x^{*})(y_{0}-x^{*})\in-\mathrm{int}(C).$$
Then, we denote by 
$$\varepsilon:=\varphi(JF(x^{*})(y_{0}-x^{*}))<0.$$
From the continuity of $JF(\cdot)$ and the compactness of $G$, we have 
$$\varphi(JF(x^{k})(y_{0}-x^{k}))\leq\frac{\varepsilon}{2},$$
for $k\in\mathcal{K}$ is sufficient large. Then, we have
\begin{align*}
	v_{\ell}(x^{k}) &= - \inf\limits_{y\in\Omega}\max\limits_{c^{*}\in G}\{\langle c^{*},JF(x^{k})(y-x^{k})\rangle+\ell D_{\omega}(y,x^{k})\}\\
	&\geq-t\varphi(JF(x^{k})(y_{0}-x^{k}))-\ell D_{\omega}(x^{k}+t(y_{0}-x^{k}),x^{k})\\
	&\geq-\frac{\varepsilon}{2}t-\ell D_{\omega}(x^{k}+t(y_{0}-x^{k}),x^{k})\\
	&\geq-\frac{\varepsilon}{2}t-\frac{\ell L_{\omega}}{2}t^{2}\|y_{0}-x^{k}\|^{2}\\
	&\geq-\frac{\varepsilon}{2}t-\ell L_{\omega}t^{2}\|y_{0}-x^{*}\|^{2},
\end{align*}
for $k\in\mathcal{K}$ is sufficient large and all $t\in(0,1)$, where the first inequality follows from $x^{k}+t(y_{0}-x^{k})\in\Omega$, the third inequality is given by $L_{\omega}$-Lipschitz continuity of $\nabla\omega(\cdot)$, and the last inequality is due to that $k\in\mathcal{K}$ is sufficient large so that, without loss of generality, we have $\|y_{0}-x^{k}\|\leq\sqrt{2}\|y_{0}-x^{*}\|$. Recall that the above inequalities holds for all $t\in(0,1)$, then there exists $t_{0}\in(0,1)$ such that $v_{\ell}(x^{k})\geq-\frac{\varepsilon t_{0}}{4}$. This contradicts $v_{\ell}(x^{k})\rightarrow0$.
Therefore, $x^{*}$ is a $C$-stationary point.
\par (iv) In the proof of assertion (ii), we obtained
$$\sum\limits_{k=0}^{\infty}D_{\omega}(x^{k+1},x^{k})\leq\frac{-\varphi(F^{*}-F(x^{0}))}
{(1+\alpha(\omega))\ell-L},$$
this combines with (\ref{e6}) yield
$$\sum\limits_{k=0}^{\infty}v_{\ell}(x^{k})\leq\frac{-\ell\varphi(F^{*}-F(x^{0}))}
{\alpha(\omega)((1+\alpha(\omega))\ell-L)}.$$
Then, assertion (iv) holds due to 
$$\sum\limits_{k=0}^{\infty}v_{\ell}(x^{k})\geq\sum\limits_{p=0}^{k-1}v_{\ell}(x^{p})\geq k\min\limits_{0\leq p\leq k-1}v_{\ell}(x^{p}).$$$\ \ \square$
\end{proof}
\begin{rem}
	In Theorem \ref{t2}(iii), if the accumulation $x^{*}$ lies in $\mathrm{bd}(\Omega)$, the continuity of $v_{\ell}(\cdot)$ and $V_{\ell}(\cdot)$ can not be derived due to $\mathrm{dom}\nabla\omega=\mathrm{int}(\Omega)$. Recall that when $\omega(\cdot)=\frac{1}{2}\|\cdot\|^{2}$, the continuity of $d(\cdot)$, corresponds to $V_{\ell}(\cdot)$, implied the stationarity of accumulation points \cite{6,vsd}. Alternatively, since the gradient of $\frac{1}{2}\|\cdot\|^{2}$ is $1$-Lipschitz continuous, the stationarity of accumulation points follows directly from a subsequence $\{d(x^{k})\}_{k\in\mathcal{K}}\rightarrow0$. The result is crucial when the continuity of $v_{\ell}(\cdot)$ and $V_{\ell}(\cdot)$ are unknown. For example, the authors used but didn't prove the continuity of $d(\cdot)$ in \cite[Theorem 3.2]{dsd}. However, the result is valid, which can be proved by $\{d(x^{k})\}_{k\in\mathcal{K}}\rightarrow0$.
\end{rem}
\subsection{Strong convergence}
In this subsection, we discuss the strong convergence of Algorithm \ref{A1} in convex case. Firstly, we present the following identity, known as three points lemma. 
\begin{lem}[Three points lemma]{\rm\cite{tp}}
	Assume that $a,b,c\in\mathrm{int(dom}\omega)$, then the following equality holds:
	$$\langle\nabla\omega(b)-\nabla\omega(a),c-a\rangle=D_{\omega}(c,a)+D_{\omega}(a,b)-D_{\omega}(c,b).$$
\end{lem}

Applying the three points lemma, we can now establish a fundamental inequality for Algorithm \ref{A1} with $C$-convex objective function.

\begin{lem}
	Assume that $F(\cdot)$ is $(L,C,e)$-smooth and $(\mu,C,e)$-strongly convex relative to $\omega(\cdot)$ for some $L>0$ and $\mu\geq0$. Let $x^{k+1}=V_{\ell}(x^{k})$ for some $\ell>0$. Then, for all $x\in\mathrm{int}(\Omega)$, we have
	\begin{equation}\label{e7}
		\langle c_{\ell}(x^{k}),F(x^{k+1})-F(x)\rangle\leq(\ell-\delta\mu)D_{\omega}(x,x^{k})-\ell D_{\omega}(x,x^{k+1})+(L-\ell)D_{\omega}(x^{k+1},x^{k}).
	\end{equation}
\end{lem}
\begin{proof}
	From the $(L,C,e)$-smoothness of $F(\cdot)$, for all $x\in\mathrm{int}(\Omega)$, we have
	\begin{align*}
	F(x^{k+1})&\preceq_{C}F(x^{k})+JF(x^{k})(x^{k+1}-x^{k})+LD_{\omega}(x^{k+1},x^{k})e\\
	&\preceq_{C}F(x)+JF(x^{k})(x^{k+1}-x)-\mu D_{\omega}(x,x^{k})e+LD_{\omega}(x^{k+1},x^{k})e,
	\end{align*}
where the second inequality is given by the $(\mu,C,e)$-strong convexity of $F(\cdot)$. Rearranging and multiplying by $c_{\ell}(x^{k})$, we obtain
	\begin{align*}
	&\langle c_{\ell}(x^{k}), F(x^{k+1})-F(x)\rangle\\
	\leq&\langle c_{\ell}(x^{k}),JF(x^{k})(x^{k+1}-x)\rangle-\delta\mu D_{\omega}(x,x^{k})+LD_{\omega}(x^{k+1},x^{k})\\
	\leq& \langle c_{\ell}(x^{k}),JF(x^{k})(x^{k+1}-x^{k})\rangle+\langle c_{\ell}(x^{k}),JF(x^{k})(x^{k}-x)\rangle-\delta\mu D_{\omega}(x,x^{k})+LD_{\omega}(x^{k+1},x^{k})\\
	=&-\ell(D_{\omega}(x^{k+1},x^{k})+D_{\omega}(x^{k},x^{k+1})) + \ell\langle\nabla\omega(x^{k})-\nabla\omega(x^{k+1}),x^{k}-x\rangle-\delta\mu D_{\omega}(x,x^{k})+LD_{\omega}(x^{k+1},x^{k})\\
	=&-\ell(D_{\omega}(x^{k+1},x^{k})+D_{\omega}(x^{k},x^{k+1}))+\ell(D_{\omega}(x,x^{k})+D_{\omega}(x^{k},x^{k+1})-D_{\omega}(x,x^{k+1}))\\
	&-\delta\mu D_{\omega}(x,x^{k})+LD_{\omega}(x^{k+1},x^{k})\\
	=&(\ell-\delta\mu)D_{\omega}(x,x^{k})-\ell D_{\omega}(x,x^{k+1})+(L-\ell)D_{\omega}(x^{k+1},x^{k}),
\end{align*}
where the first equality follows from Proposition \ref{p1}(iii) and equality (\ref{e4}), the second equality is given by three points lemma with $a=x^{k}$, $b=x^{k+1}$ and $c=x$.$\ \ \square$
\end{proof}

Before presenting the convergence result in convex case, we recall the following result on nonnegative sequences.
\begin{lem}\label{l4}{\rm\cite[Lemma 2]{qf}}
	Let $\{a^{k}\}$ and $\{b^{k}\}$ be nonnegative sequences. Assume that $\sum\limits_{k=1}^{\infty}b^{k}<\infty$ and 
	$$a^{k+1}\leq a^{k}+b^{k}.$$
	Then, $\{a^{k}\}$ converges.
\end{lem}

We also recall the following assumption.
\begin{assumption}\cite[Assumption H]{lc}\label{a2}
	Assume the Legendre function $\omega(\cdot)$ satisfies
	\begin{itemize}
		\item[$\mathrm{(i)}$] If $\{x^{k}\}\subset\mathrm{int}(\Omega)$ converges to some point $x\in\Omega$, then $D_{\omega}(x,x^{k})\rightarrow0$.
		\item[$\mathrm{(ii)}$] Reciprocally, if $x\in\Omega$ and if $\{x^{k}\}\subset\mathrm{int}(\Omega)$ is such that $D_{\omega}(x,x^{k})\rightarrow0$, then $x^{k}\rightarrow x$.
	\end{itemize}
\end{assumption}
\begin{rem}
	Note that $\omega(\cdot)$ is differentiable on $\mathrm{int}(\Omega)$, Assumption \ref{a2}(i) holds directly for $x\in\mathrm{int}(\Omega)$. On the other hand, if $x\in\mathrm{int}(\Omega)$, Assumption \ref{a2}(ii) follows from the strict convexity of $\omega(\cdot)$ on $\mathrm{int}(\Omega)$. Hence, Assumption \ref{a2} is proposed for $x\in\mathrm{bd}(\Omega)$. 
\end{rem}

In what follows, we give the convergence result for convex case.
\begin{theo}\label{t3}
	The assumptions of Theorem \ref{t2} hold and $F(\cdot)$ is $C$-convex. Then, we have
	\begin{itemize}
		\item[$\mathrm{(i)}$] there exists $F^{*}\preceq_{C}F(x^{k})$ for all $k$ such that $\lim\limits_{k\rightarrow\infty}F(x^{k})=F^{*}$.
		\item[$\mathrm{(ii)}$] $\sum\limits_{k=0}^{\infty}D_{\omega}(x^{k+1},x^{k})<\infty$.
		\item[$\mathrm{(iii)}$] every accumulation point of $\{x_{k}\}$ is a weakly efficient solution. Assume that Assumption \ref{a2} holds, then $\{x_{k}\}$ converges to some weakly efficient solution $x^{*}$.
		\item[$\mathrm{(iv)}$] $\langle \bar{c}_{\ell}^{k-1},F(x^{k})-F^{*}\rangle\leq \frac{\ell D_{\omega}(x^{*},x^{0})+(L-\ell)\sum\limits_{p=0}^{k-1}D_{\omega}(x^{p+1},x^{p})}{k}$, where $\bar{c}_{\ell}^{k-1}=\frac{1}{k}\sum\limits_{p=0}^{k-1}c_{\ell}(x^{p})$.
		\item[$\mathrm{(v)}$] if $\ell\geq L$, then $\langle \bar{c}_{\ell}^{k-1},F(x^{k})-F^{*}\rangle\leq \frac{\ell D_{\omega}(x^{*},x^{0})}{k}$, where $\bar{c}_{\ell}^{k-1}=\frac{1}{k}\sum\limits_{p=0}^{k-1}c_{\ell}(x^{p})$.
	\end{itemize}
\end{theo}
\begin{proof}
	Since all assumptions of Theorem \ref{t2} hold, we obtain assertions (i) and (ii).
	\par (iii) Denote by $X^{*}_{F}$=$\{x\in\Omega:F(x)=F^{*}\}$. From the proof of Theorem \ref{t2}(iii), there exists $x^{*}\in X^{*}_{F}$ is an accumulation point of $\{x^{k}\}$. Next, we prove $x^{*}$ is a weakly efficient point. Suppose by the contrary that there exists $\hat{x}\in\Omega$ such that $F(\hat{x})\prec_{C}F(x^{*})$. 
	 Substituting $x=\hat{x}$ and $\mu=0$ into inequality (\ref{e7}), we have
	\begin{equation}\label{e8}
		\langle c_{\ell}(x^{k}),F(x^{k+1})-F(\hat{x})\rangle\leq\ell(D_{\omega}(\hat{x},x^{k})-D_{\omega}(\hat{x},x^{k+1}))+(L-\ell)D_{\omega}(x^{k+1},x^{k}).
	\end{equation}
Summing the above inequality over $0$ to $k-1$, we obtain
\begin{align*}
	\ell(D_{\omega}(\hat{x},x^{0})-D_{\omega}(\hat{x},x^{k}))+(L-\ell)\sum\limits_{p=0}^{k-1}D_{\omega}(x^{p+1},x^{p})&\geq\langle \sum\limits_{p=0}^{k-1}c_{\ell}(x^{p}),F(x^{p+1})-F(\hat{x})\rangle\\
	&\geq\langle \sum\limits_{p=0}^{k-1}c_{\ell}(x^{p}),F(x^{*})-F(\hat{x})\rangle,
\end{align*}
where the second inequality follows from $F(x^{*})\preceq_{C}F(x^{p})$. Multiplying by $\frac{1}{k}$, we obtain  
\begin{equation}\label{eq19}
	\langle \frac{1}{k}\sum\limits_{p=0}^{k-1}c_{\ell}(x^{p}),F(x^{*})-F(\hat{x})\rangle\leq \frac{\ell D_{\omega}(\hat{x},x^{0})+(L-\ell)\sum\limits_{p=0}^{k-1}D_{\omega}(x^{p+1},x^{p})}{k}.
\end{equation}
Since $\frac{1}{k}\sum\limits_{p=0}^{k-1}c_{\ell}(x^{p})\in G$, this together with $F(\hat{x})\prec_{C}F(x^{*})$ imply 
\begin{eqnarray}\label{eq20}
	\langle \frac{1}{k}\sum\limits_{p=0}^{k-1}c_{\ell}(x^{p}),F(x^{*})-F(\hat{x})\rangle\geq \delta_{0},
\end{eqnarray}
where $\delta_{0}:=\min\limits_{c^{*}\in G}\langle c^{*},F(x^{*})-F(\hat{x})\rangle>0$. However, as $k\rightarrow\infty$, the right hand side of inequality (\ref{eq19}) tends to $0$ due to $\sum\limits_{k=0}^{\infty}D_{\omega}(x^{k+1},x^{k})<\infty$. This contradicts ineuqality (\ref{eq20}). Therefore, $x^{*}$ is a weakly efficient point. In what follows, we prove the whole sequence converges to $x^{*}$. Substituting $x=x^{*}$ and $\mu=0$ into inequality (\ref{e7}), we have
\begin{equation}\label{eq21}
	\langle c_{\ell}(x^{k}),F(x^{k+1})-F(x^{*})\rangle\leq\ell(D_{\omega}(x^{*},x^{k})-D_{\omega}(x^{*},x^{k+1}))+(L-\ell)D_{\omega}(x^{k+1},x^{k}).
\end{equation}
The left hand side of the above inequality is nonnegative due to $F^{*}\preceq_{C}F(x^{k+1})$, then
$$D_{\omega}(x^{*},x^{k+1})\leq D_{\omega}(x^{*},x^{k})+\frac{L-\ell}{\ell}D_{\omega}(x^{k+1},x^{k}).$$ 
This combines with $\sum\limits_{k=0}^{\infty}D_{\omega}(x^{k+1},x^{k})<\infty$ and Lemma \ref{l4} imply
$\{D_{\omega}(x^{*},x^{k})\}$ converges. On the other hand, $x^{*}$ is an accumulation point of $\{x^{k}\}$. It follows from Assumption \ref{a2}(i) that $\{D_{\omega}(x^{*},x^{k})\}$ converges to $0$. Therefore, by Assumption \ref{a2}(ii), we obtain sequence $\{x^{k}\}$ converges to $x^{*}$.
\par (iv) Summing inequality (\ref{eq21}) over $0$ to $k-1$, we obtain
\begin{align*}
	\ell(D_{\omega}(x^{*},x^{0})-D_{\omega}(x^{*},x^{k}))+(L-\ell)\sum\limits_{p=0}^{k-1}D_{\omega}(x^{p+1},x^{p})&\geq\langle \sum\limits_{p=0}^{k-1}c_{\ell}(x^{p}),F(x^{p+1})-F^{*}\rangle\\
	&\geq\langle \sum\limits_{p=0}^{k-1}c_{\ell}(x^{p}),F(x^{k})-F^{*}\rangle,
\end{align*}
where the second inequality follows from $F(x^{p+1})\preceq_{C}F(x^{p})$. Multiplying by $\frac{1}{k}$, we obtain assertion (iv). 
\par (v) Assertion (v) follows directly from assertion (iv).$\ \ \square$
\end{proof}
\subsection{Linear convergence}
This subsection is devoted to discuss the linear convergence of Algorithm \ref{A1}. Fristly, we present the following results under the assumption that $F(\cdot)$ is $(\mu,C,e)$-strongly convex relative to $\omega(\cdot)$ for $\mu>0$.
\begin{theo}\label{t4}
		The assumptions of Theorem \ref{t2} hold and $F(\cdot)$ is $(\mu,C,e)$-strongly convex relative to $\omega(\cdot)$ for some $\mu>0$. Then, we have
	\begin{itemize}
		\item[$\mathrm{(i)}$] there exists $F^{*}\preceq_{C}F(x^{k})$ for all $k$ such that $\lim\limits_{k\rightarrow\infty}F(x^{k})=F^{*}$.
		\item[$\mathrm{(ii)}$] $\sum\limits_{k=0}^{\infty}D_{\omega}(x^{k+1},x^{k})<\infty$.
		\item[$\mathrm{(iii)}$] every accumulation point of $\{x_{k}\}$ is a weakly efficient solution. Suppose that Assumption \ref{a2} holds, then $\{x_{k}\}$ converges to some weakly efficient solution $x^{*}$.
		\item[$\mathrm{(iv)}$] if $\ell\geq L$, then $D_{\omega}(x^{*},x^{k+1})\leq\frac{\ell-\delta\mu}{\ell}
		D_{\omega}(x^{*},x^{k})$.
	\end{itemize}
\end{theo}
\begin{proof}
	Since $(\mu,C,e)$-strong convexity is sharper than $C$-convexity, the assertions (i)-(iii) hold directly from Theorem \ref{t3}.
	\par (iv) Substituting $x=x^{*}$ into inequality (\ref{e7}), we have
	$$\langle c_{\ell}(x^{k}),F(x^{k+1})-F^{*}\rangle\leq(\ell-\delta\mu)D_{\omega}(x^{*},x^{k})-\ell D_{\omega}(x^{*},x^{k+1})+(L-\ell)D_{\omega}(x^{k+1},x^{k}).$$ 
	This together with $F^{*}\preceq_{C}F(x^{k+1})$ and $\ell\geq L$ yield
	$$D_{\omega}(x^{*},x^{k+1})\leq\frac{\ell-\delta\mu}{\ell}
	D_{\omega}(x^{*},x^{k}).$$
	$\ \ \square$
\end{proof}

Next, we show the Q-linear convergence result of $\{u_{0}(x^{k})\}$ with vector Bregman-PL inequality.
\begin{theo}
	Assume that vector Bregman-PL inequality (\ref{eb}) holds and $F(\cdot)$ is $(L,C,e)$-smooth relative to $\omega(\cdot)$ for some $L>0$. If $\ell\geq L$, then the sequence $\{x^{k}\}$ generated by Algorithm \ref{A1} satisfies
	$$u_{0}(x^{k+1})\leq\left(1-\frac{1}{\tau(\ell)}\right)u_{0}(x^{k}).$$
\end{theo}
\begin{proof}
	Since $F(\cdot)$ is $(L,C,e)$-smooth relative to $\omega(\cdot)$ and $\ell\geq L$, we have
	$$F(x^{k+1})\preceq_{C}F(x^{k})+JF(x^{k})(x^{k+1}-x^{k})+\ell D_{\omega}(x^{k+1},x^{k})e.$$
Then, for all $x\in\Omega$ we obtain
$$F(x^{k+1})-F(x)\preceq_{C}F(x^{k})-F(x)+JF(x^{k})(x^{k+1}-x^{k})+\ell D_{\omega}(x^{k+1},x^{k})e.$$
It follows that
\begin{align*}
	&\sup\limits_{x\in\Omega}\min\limits_{c^{*}\in G}\langle c^{*},F(x^{k})-F(x)\rangle\\
	\geq&\sup\limits_{x\in\Omega}\min\limits_{c^{*}\in G}\{\langle c^{*},F(x^{k+1})-F(x)+JF(x^{k})(x^{k}-x^{k+1})-\ell D_{\omega}(x^{k+1},x^{k})e\rangle\}\\
	\geq&\sup\limits_{x\in\Omega}\{\min\limits_{c^{*}\in G}\langle c^{*},F(x^{k+1})-F(x)\rangle+\min\limits_{c^{*}\in G}\langle c^{*},JF(x^{k})(x^{k}-x^{k+1})\rangle-\ell D_{\omega}(x^{k+1},x^{k})\}\\
	=&\sup\limits_{x\in\Omega}\min\limits_{c^{*}\in G}\langle c^{*},F(x^{k+1})-F(x)\rangle +v_{\ell}(x^{k}),
\end{align*}
Hence,
\begin{align*}
 u_{0}(x^{k+1})&\leq u_{0}(x^{k})-v_{\ell}(x^{k})\\
 &\leq\left(1-\frac{1}{\tau(\ell)}\right)u_{0}(x^{k}).
\end{align*}
From (\ref{evu}), (\ref{eb}) and $\ell\geq L$, we have $\tau(\ell)\geq1$. Then $0<1-\frac{1}{\tau(\ell)}<1$, we derive that $\{u_{0}(x^{k})\}$ converges linearly.$\ \ \square$
\end{proof}

\section{Conclusions}
We extended the relative smoothness and relative strong convexity to vector-valued functions in the context of VOPs. Then we presented convergence rates for the interior Bregman gradient method using the generalized assumptions. In addition, we also devised a vector Bregman-PL inequality, which was used to prove linear convergence of the proposed method. It is worth noting that the convergence results extended the ones in \cite{com} to non-Lipschitz gradient continuous and non-strongly convex case.

\end{document}